\documentclass[a4paper,onecolumn, superscriptaddress,10pt,accepted=2025-10-03,issue=2, volume=7, shorttitle=papers]{compositionalityarticle}

 \pdfoutput=1
\usepackage[utf8]{inputenc}
\usepackage[english]{babel}
\usepackage[T1]{fontenc}
\usepackage{amsmath}

\usepackage{tikz}
\usepackage{lipsum}

\usepackage[bookmarks]{hyperref}

\usepackage{amssymb}
\usepackage{latexsym}
\usepackage{amsthm}
\usepackage{graphicx}
\usepackage{bbm}
\usepackage{verbatim}
\usepackage{epigraph}
\usepackage{mathtools}
\usepackage{authblk}
\usepackage[pdflatex]{crop}
\usepackage{tikz-cd}
\usepackage{wasysym}
\usepackage{adjustbox}
\usepackage{microtype}
\usepackage{chngcntr}
\usepackage{cleveref}
\allowdisplaybreaks
\usetikzlibrary{shapes.geometric,fit}

\usepackage{xcolor}
\definecolor{fillcolor}{HTML}{c0d3e1} 

\newcommand{\auth}[0]{Paolo Perrone}
\newcommand{\tit}[0]{Lifting couplings in Wasserstein spaces}
\newcommand{\kw}[0]{}

\hypersetup{
pdfauthor={\auth},%
pdftitle={\tit},%
colorlinks, linktocpage=true, pdfstartpage=1, pdfstartview=FitV,%
breaklinks=true, pdfpagemode=UseNone, pageanchor=true, pdfpagemode=UseOutlines,%
plainpages=false, bookmarksnumbered, bookmarksopen=true, bookmarksopenlevel=1,%
hypertexnames=true, pdfhighlight=/O,%
urlcolor=black, linkcolor=black, citecolor=black, 
}
\numberwithin{equation}{section}

\theoremstyle{plain}
\newtheorem{thm}{Theorem}[section]
\newtheorem{lemma}[thm]{Lemma}
\newtheorem{prop}[thm]{Proposition}
\newtheorem{cor}[thm]{Corollary}

\newtheorem{deph}[thm]{Definition}

\theoremstyle{definition}
\newtheorem{remark}[thm]{Remark}
\newtheorem{eg}[thm]{Example}

\DeclareMathOperator{\Hom}{Hom}

\newcommand{\R}{\mathbb{R}}

\DeclareFontFamily{OT1}{pzc}{}
\DeclareFontShape{OT1}{pzc}{m}{it}{<-> s * [1.10] pzcmi7t}{}
\DeclareMathAlphabet{\mathpzc}{OT1}{pzc}{m}{it}

\mathchardef\mh="2D
 \newcommand{\cat}[1]{{\mathpzc{#1}}} 
\newcommand{\ar}[2][]{\arrow{#2}{#1}}

\newcommand{\op}{\mathrm{op}} 

\newcommand{\e}{\varepsilon}

\newcommand{\id}{\mathrm{id}} 

\newcommand{\intersection}{\cap}

\newcommand{\cost}{\mathrm{cost}}
\DeclareMathOperator*{\Opt}{Opt}

\makeatletter
\DeclareRobustCommand{\cev}[1]{%
  \mathpalette\do@cev{#1}%
}
\newcommand{\do@cev}[2]{%
  \fix@cev{#1}{+}%
  \reflectbox{$\m@th#1\vec{\reflectbox{$\fix@cev{#1}{-}\m@th#1#2\fix@cev{#1}{+}$}}$}%
  \fix@cev{#1}{-}%
}
\newcommand{\fix@cev}[2]{%
  \ifx#1\displaystyle
    \mkern#23mu
  \else
    \ifx#1\textstyle
      \mkern#23mu
    \else
      \ifx#1\scriptstyle
        \mkern#22mu
      \else
        \mkern#22mu
      \fi
    \fi
  \fi
}
\makeatother

\let\originalleft\left
\let\originalright\right
\renewcommand{\left}{\mathopen{}\mathclose\bgroup\originalleft}
\renewcommand{\right}{\aftergroup\egroup\originalright}

\renewcommand{\theenumi}{(\alph{enumi})}

\begin{document}
 \title{Lifting Couplings in Wasserstein Spaces}
 \date{}
 \author{Paolo Perrone}
 \email{paolo.perrone@cs.ox.ac.uk}
 \homepage{www.paoloperrone.org}
 \orcid{0000-0002-9123-9089}
 \affiliation{University of Oxford, Computer Science Department,
 	Oxford, UK}
 	
\newcommand{\authorsforheader}{Perrone, Paolo}
\newcommand{\paperdoi}{10.46298/compositionality-7-2}
\newcommand{\receiveddate}{2022-04-01}
\newcommand{\accepteddate}{2025-03-10}
 
\begin{abstract}
 This paper makes mathematically precise the idea that conditional probabilities are analogous to \emph{path liftings} in geometry.  
 
 The idea of lifting is modeled in terms of the category-theoretic concept of a \emph{lens}, which can be interpreted as a consistent choice of arrow liftings. The category we study is the one of probability measures over a given standard Borel space, with morphisms given by the couplings, or transport plans.
 
 The geometrical picture is even more apparent once we equip the arrows of the category with \emph{weights}, which one can interpret as ``lengths'' or ``costs'', forming a so-called \emph{weighted category}, which unifies several concepts of category theory and metric geometry. 
 Indeed, we show that the weighted version of a lens is tightly connected to the notion of \emph{submetry} in geometry.
 
 Every weighted category gives rise to a pseudo-quasimetric space via optimization over the arrows. In particular, Wasserstein spaces can be obtained from the weighted categories of probability measures and their couplings, with the weight of a coupling given by its cost. 
 In this case, conditional distributions allow one to form weighted lenses, which one can interpret as ``lifting transport plans, while preserving their cost''. 
\end{abstract}

\section{Introduction}

This paper explores a connection between conditional probabilities and liftings in geometry and topology. 

\paragraph{Projections and liftings.}
Forming a marginal of a probability measure is a destructive operation. Given a probability measure on a product space $X\times Y$, its marginal on $X$ is often called a \emph{projection}, borrowing from the language of geometry. The idea is that the marginal distribution captures only the variability in $X$, while the original joint measure may exhibit also variability in $Y$ (as well as possible correlations), in a way that's analogous to the projection onto the $X$-axis of a geometrical figure in the $XY$-plane. 
Similar intuitions can be had, more generally, whenever a probability measure is pushed forward along a noninjective measurable map, or when one restricts the measure to a coarser $\sigma$-algebra. 

There are several situations, in geometry and related fields, where one has a projection $p:E\to B$, and one wants to ``translate motions in $B$ to corresponding motions in $E$''. 
This idea is made mathematically precise in different ways, such as the following.
\begin{itemize}
 \item In algebraic topology, a \emph{fibration} $p:E\to B$ satisfies in particular a path-lifting property. Given a point $e\in E$ and a curve $c$ in $B$ starting at $p(e)$, we can lift $c$ to a curve $\tilde c$ in $E$ starting at $e$ which is projected by $p$ back to $c$. 
 \item In metric geometry, a \emph{submetry} $p:E\to B$ can be seen as ``lifting distances''. Consider a point $e\in E$, and a point $b\in B$ which has distance $r$ from $p(e)$. Then we can lift $b$ to a point of $E$ which has distance $r$ from $e$. (See also \Cref{lenses}.)
 \item In differential topology, a \emph{bundle} may be equipped with a \emph{connection}, which can be interpreted as a way to lift vectors, or also paths (in an open neighborhood), in a way which preserves concatenation. 
\end{itemize}
We are particularly interested in liftings that \emph{preserve composition} (of curves, displacements, etc.). These can be effectively modeled via the category-theoretic notion of a \emph{lens}, which we recall in \Cref{lenses}. Intuitively, a lets consists of a functor $F:\cat{E}\to \cat{B}$ together with a ``lifting'' assignment which, in analogy to the points above, takes an object $E$ of $\cat{E}$ and a morphism of $\cat{B}$ with domain $F(E)$, and gives a morphism of $\cat{E}$ with domain $E$, in such a way to preserve identity and composition. (See \Cref{lenses} for the details.)

The main focus of this work is the idea that by means of conditional distributions, one can construct a similar lifting to the ones described above. 
Namely, take a surjective, measurable function between measurable spaces $f:X\to Y$, and consider a probability measure $p$ on $X$. Take the pushforward measure $f_\sharp p$ on $Y$, and consider a transport plan (a.k.a.~coupling) from $f_\sharp p$ to another measure $q$ on $Y$. Under what conditions can we lift the transport plan on $Y$ to a transport plan on $X$?
And when does this notion of lifting preserve the composition of transport plans?

We will show that these liftings exist whenever $f$ is equipped with the structure of a \emph{measurable lens} (\Cref{main}). This in particular includes the case of product projections, which at the level of probability measures give marginals -- and in that case, the lifting is given by forming a particular conditional product (\Cref{prodlens}).

\paragraph{Optimal transport and couplings.}
Consider two probability distributions $p$ and $q$ on a measurable space $X$. 
A \emph{coupling} or \emph{transport plan} between $p$ and $q$ is a joint probability distribution $s$ on $X\times X$ such that its marginals are $p$ and $q$. 
The name ``transport plan'' is motivated by the intuition that $s$ can be considered a specified way of moving the mass of $p$ to the configuration $q$. 
This is most visible in the discrete case. So, for an illustration, suppose that $X$ is a finite set. A probability distribution $p$ on $X$ amounts to a function $p:X\to[0,1]$ such that 
\[
 \sum_{x\in X} p(x) = 1 .
\]
So let $p$ and $q$ be distributions on $X$. A coupling $s$ of $p$ and $q$ is then a function $s:X\times X\to[0,1]$ such that
\[
 \sum_{x'\in X} s(x,x') = p(x) \qquad\mbox{and}\qquad \sum_{x'\in X} s(x',x) = q(x)
\]
for all $x\in X$.
For example, if $X=\{x,y\}$, consider the probability measures $p$ and $q$ given as follows.
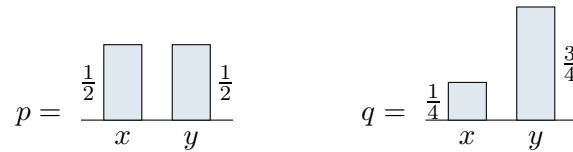
\[
 p = 
 \begin{tikzpicture}[baseline,
 x={(0:1cm)},y={(90:1cm)},z={(0:1cm)}]
  \draw (0,0,-2) -- (2,0,-2) ;
  \draw [fill=fillcolor!50] (0.55,0,-2) -- ++(0.25,0,0) -- ++(0,1,0) -- ++(-0.5,0,0) -- ++(0,-1,0) -- ++(0.25,0,0);
  \node at (0.55,-0.2,-2) {$x\strut$};
  \node at (0.1,0.5,-2) {$\frac{1}{2}$};
  \draw [fill=fillcolor!50] (1.45,0,-2) -- ++(0.25,0,0) -- ++(0,1,0) -- ++(-0.5,0,0) -- ++(0,-1,0) -- ++(0.25,0,0);
  \node at (1.45,-0.2,-2) {$y\strut$};
  \node at (1.9,0.5,-2) {$\frac{1}{2}$};
 \end{tikzpicture}
 \qquad\qquad
 q = 
 \begin{tikzpicture}[baseline,
 x={(0:1cm)},y={(90:1cm)},z={(0:1cm)}]
  \draw (4,0,0) -- (4,0,2) ;
  \draw [fill=fillcolor!50] (4,0,0.55) -- ++(0,0,0.25) -- ++(0,0.5,0) -- ++(0,0,-0.5) -- ++(0,-0.5,0) -- ++(0,0,0.25);
  \node at (4,-0.2,0.55) {$x\strut$};
  \node at (4,0.25,0.1) {$\frac{1}{4}$};
  \draw [fill=fillcolor!50] (4,0,1.45) -- ++(0,0,0.25) -- ++(0,1.5,0) -- ++(0,0,-0.5) -- ++(0,-1.5,0) -- ++(0,0,0.25);
  \node at (4,-0.2,1.45) {$y\strut$};
  \node at (4,0.75,1.9) {$\frac{3}{4}$};
 \end{tikzpicture}
\]

Here is a possible coupling between $p$ and $q$.
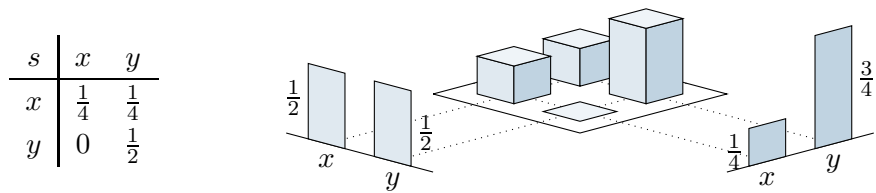
\[
\begin{tabular}{c|cc}
 $s$ & $x$ & $y$ \\
 \hline
 $x$ & $\frac{1}{4}$ & $\frac{1}{4}$ \\
 $y$ & $0$ & $\frac{1}{2}$
\end{tabular}
\qquad\qquad
\begin{tikzpicture}[baseline={(current bounding box.center)},
    x={(-15:1cm)},y={(90:1cm)},z={(15:1cm)}]
  \draw (0,0,0) -- (2,0,0) -- (2,0,2) -- (0,0,2) -- (0,0,0);
  
  \draw [dotted] (0.55,0,2) -- (0.55,0,-2);
  \draw [dotted] (1.45,0,2) -- (1.45,0,-2);
  \draw [dotted] (0,0,0.55) -- (4,0,0.55);
  \draw [dotted] (0,0,1.45) -- (4,0,1.45);
 
  \draw (0,0,-2) -- (2,0,-2) ;
  \draw [fill=fillcolor!50] (0.55,0,-2) -- ++(0.25,0,0) -- ++(0,1,0) -- ++(-0.5,0,0) -- ++(0,-1,0) -- ++(0.25,0,0);
  \node at (0.55,-0.2,-2) {$x\strut$};
  \node at (0.1,0.5,-2) {$\frac{1}{2}$};
  \draw [fill=fillcolor!50] (1.45,0,-2) -- ++(0.25,0,0) -- ++(0,1,0) -- ++(-0.5,0,0) -- ++(0,-1,0) -- ++(0.25,0,0);
  \node at (1.45,-0.2,-2) {$y\strut$};
  \node at (1.9,0.5,-2) {$\frac{1}{2}$};
 
  \draw (4,0,0) -- (4,0,2) ;
  \draw [fill=fillcolor] (4,0,0.55) -- ++(0,0,0.25) -- ++(0,0.5,0) -- ++(0,0,-0.5) -- ++(0,-0.5,0) -- ++(0,0,0.25);
  \node at (4,-0.2,0.55) {$x\strut$};
  \node at (4,0.25,0.1) {$\frac{1}{4}$};
  \draw [fill=fillcolor] (4,0,1.45) -- ++(0,0,0.25) -- ++(0,1.5,0) -- ++(0,0,-0.5) -- ++(0,-1.5,0) -- ++(0,0,0.25);
  \node at (4,-0.2,1.45) {$y\strut$};
  \node at (4,0.75,1.9) {$\frac{3}{4}$};
  
  \draw [fill=fillcolor!50] (0.55,0,1.45)++(0,0,-0.25) -- ++(0.25,0,0) -- ++(0,0.5,0) -- ++(-0.5,0,0) -- ++(0,-0.5,0) -- ++(0.25,0,0);
  \draw [fill=fillcolor] (0.55,0,1.45)++(0.25,0,0) -- ++(0,0,0.25) -- ++(0,0.5,0) -- ++(0,0,-0.5) -- ++(0,-0.5,0) -- ++(0,0,0.25);
  \draw [fill=fillcolor!25] (0.55,0.5,1.45)++(0.25,0,-0.25) -- ++(-0.5,0,0) -- ++(0,0,0.5) -- ++(0.5,0,0) -- ++(0,0,-0.5);
  
  \draw [fill=fillcolor!50] (0.55,0,0.55)++(0,0,-0.25) -- ++(0.25,0,0) -- ++(0,0.5,0) -- ++(-0.5,0,0) -- ++(0,-0.5,0) -- ++(0.25,0,0);
  \draw [fill=fillcolor] (0.55,0,0.55)++(0.25,0,0) -- ++(0,0,0.25) -- ++(0,0.5,0) -- ++(0,0,-0.5) -- ++(0,-0.5,0) -- ++(0,0,0.25);
  \draw [fill=fillcolor!25] (0.55,0.5,0.55)++(0.25,0,-0.25) -- ++(-0.5,0,0) -- ++(0,0,0.5) -- ++(0.5,0,0) -- ++(0,0,-0.5);
  
  \draw [fill=fillcolor!50] (1.45,0,1.45)++(0,0,-0.25) -- ++(0.25,0,0) -- ++(0,1,0) -- ++(-0.5,0,0) -- ++(0,-1,0) -- ++(0.25,0,0);
  \draw [fill=fillcolor] (1.45,0,1.45)++(0.25,0,0) -- ++(0,0,0.25) -- ++(0,1,0) -- ++(0,0,-0.5) -- ++(0,-1,0) -- ++(0,0,0.25);
  \draw [fill=fillcolor!25] (1.45,1,1.45)++(0.25,0,-0.25) -- ++(-0.5,0,0) -- ++(0,0,0.5) -- ++(0.5,0,0) -- ++(0,0,-0.5);
  
  \draw [fill=fillcolor!25] (1.45,0,0.55)++(0.25,0,-0.25) -- ++(-0.5,0,0) -- ++(0,0,0.5) -- ++(0.5,0,0) -- ++(0,0,-0.5);
\end{tikzpicture}
\]
We can interpret this coupling as a transport plan as follows, following the picture from left to right:
\begin{itemize}
 \item Split the mass at $x$ (which is $\tfrac{1}{2}$) into $\frac{1}{4}$ and $\frac{1}{4}$, keep $\frac{1}{4}$ at $x$, and move the rest to $y$;
 \item Leave all the mass at $y$ (which is $\tfrac{1}{2}$) in place.
\end{itemize}
Here is another possible coupling. 
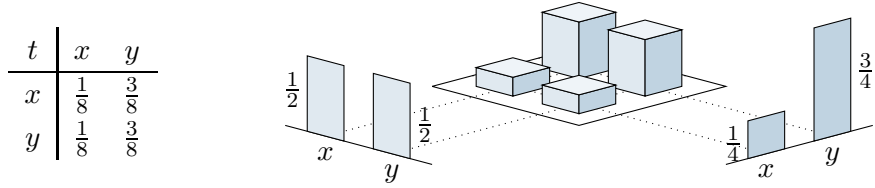
\[
\begin{tabular}{c|cc}
 $t$ & $x$ & $y$ \\
 \hline
 $x$ & $\frac{1}{8}$ & $\frac{3}{8}$ \\
 $y$ & $\frac{1}{8}$ & $\frac{3}{8}$
\end{tabular}
\qquad\qquad
\begin{tikzpicture}[baseline={(current bounding box.center)},
    x={(-15:1cm)},y={(90:1cm)},z={(15:1cm)}]
  \draw (0,0,0) -- (2,0,0) -- (2,0,2) -- (0,0,2) -- (0,0,0);
  
  \draw [dotted] (0.55,0,2) -- (0.55,0,-2);
  \draw [dotted] (1.45,0,2) -- (1.45,0,-2);
  \draw [dotted] (0,0,0.55) -- (4,0,0.55);
  \draw [dotted] (0,0,1.45) -- (4,0,1.45);
 
  \draw (0,0,-2) -- (2,0,-2) ;
  \draw [fill=fillcolor!50] (0.55,0,-2) -- ++(0.25,0,0) -- ++(0,1,0) -- ++(-0.5,0,0) -- ++(0,-1,0) -- ++(0.25,0,0);
  \node at (0.55,-0.2,-2) {$x\strut$};
  \node at (0.1,0.5,-2) {$\frac{1}{2}$};
  \draw [fill=fillcolor!50] (1.45,0,-2) -- ++(0.25,0,0) -- ++(0,1,0) -- ++(-0.5,0,0) -- ++(0,-1,0) -- ++(0.25,0,0);
  \node at (1.45,-0.2,-2) {$y\strut$};
  \node at (1.9,0.5,-2) {$\frac{1}{2}$};
 
  \draw (4,0,0) -- (4,0,2) ;
  \draw [fill=fillcolor] (4,0,0.55) -- ++(0,0,0.25) -- ++(0,0.5,0) -- ++(0,0,-0.5) -- ++(0,-0.5,0) -- ++(0,0,0.25);
  \node at (4,-0.2,0.55) {$x\strut$};
  \node at (4,0.25,0.1) {$\frac{1}{4}$};
  \draw [fill=fillcolor] (4,0,1.45) -- ++(0,0,0.25) -- ++(0,1.5,0) -- ++(0,0,-0.5) -- ++(0,-1.5,0) -- ++(0,0,0.25);
  \node at (4,-0.2,1.45) {$y\strut$};
  \node at (4,0.75,1.9) {$\frac{3}{4}$};
  
  \draw [fill=fillcolor!50] (0.55,0,1.45)++(0,0,-0.25) -- ++(0.25,0,0) -- ++(0,0.75,0) -- ++(-0.5,0,0) -- ++(0,-0.75,0) -- ++(0.25,0,0);
  \draw [fill=fillcolor] (0.55,0,1.45)++(0.25,0,0) -- ++(0,0,0.25) -- ++(0,0.75,0) -- ++(0,0,-0.5) -- ++(0,-0.75,0) -- ++(0,0,0.25);
  \draw [fill=fillcolor!25] (0.55,0.75,1.45)++(0.25,0,-0.25) -- ++(-0.5,0,0) -- ++(0,0,0.5) -- ++(0.5,0,0) -- ++(0,0,-0.5);
  
  \draw [fill=fillcolor!50] (0.55,0,0.55)++(0,0,-0.25) -- ++(0.25,0,0) -- ++(0,0.25,0) -- ++(-0.5,0,0) -- ++(0,-0.25,0) -- ++(0.25,0,0);
  \draw [fill=fillcolor] (0.55,0,0.55)++(0.25,0,0) -- ++(0,0,0.25) -- ++(0,0.25,0) -- ++(0,0,-0.5) -- ++(0,-0.25,0) -- ++(0,0,0.25);
  \draw [fill=fillcolor!25] (0.55,0.25,0.55)++(0.25,0,-0.25) -- ++(-0.5,0,0) -- ++(0,0,0.5) -- ++(0.5,0,0) -- ++(0,0,-0.5);
  
  \draw [fill=fillcolor!50] (1.45,0,1.45)++(0,0,-0.25) -- ++(0.25,0,0) -- ++(0,0.75,0) -- ++(-0.5,0,0) -- ++(0,-0.75,0) -- ++(0.25,0,0);
  \draw [fill=fillcolor] (1.45,0,1.45)++(0.25,0,0) -- ++(0,0,0.25) -- ++(0,0.75,0) -- ++(0,0,-0.5) -- ++(0,-0.75,0) -- ++(0,0,0.25);
  \draw [fill=fillcolor!25] (1.45,0.75,1.45)++(0.25,0,-0.25) -- ++(-0.5,0,0) -- ++(0,0,0.5) -- ++(0.5,0,0) -- ++(0,0,-0.5);
  
  \draw [fill=fillcolor!50] (1.45,0,0.55)++(0,0,-0.25) -- ++(0.25,0,0) -- ++(0,0.25,0) -- ++(-0.5,0,0) -- ++(0,-0.25,0) -- ++(0.25,0,0);
  \draw [fill=fillcolor] (1.45,0,0.55)++(0.25,0,0) -- ++(0,0,0.25) -- ++(0,0.25,0) -- ++(0,0,-0.5) -- ++(0,-0.25,0) -- ++(0,0,0.25);
  \draw [fill=fillcolor!25] (1.45,0.25,0.55)++(0.25,0,-0.25) -- ++(-0.5,0,0) -- ++(0,0,0.5) -- ++(0.5,0,0) -- ++(0,0,-0.5);
\end{tikzpicture}
\]
Again, we can interpret this as follows:
\begin{itemize}
 \item Split the mass at $x$ (which is $\tfrac{1}{2}$) into $\frac{1}{8}$ and $\frac{3}{8}$, keep $\frac{1}{8}$ at $x$, and move $\frac{3}{8}$ to $y$;
 \item Split the mass at $x$ (which is $\tfrac{1}{2}$) into $\frac{1}{8}$ and $\frac{3}{8}$, move $\frac{1}{8}$ to $x$, and keep $\frac{3}{8}$ at $y$.
\end{itemize}
We notice a few things.
\begin{enumerate}
 \item A transport plan from $p$ to $q$ can also be read backwards, as a transport plan from $q$ to $p$.
 \item Transport plans can be composed: provided some technical conditions are satisfied, they form a category (see \Cref{catcouplings}). Since a transport plan can be read in both directions, this category has a natural dagger structure.
 \item The first example, intuitively, ``moves less mass'' than the second example (it only moves $\tfrac{1}{4}$ from $x$ to $y$, versus the second example that moves $\tfrac{3}{8}$ from $y$ to $x$ and $\tfrac{1}{8}$ from $y$ back to $x$, for a total of $\tfrac{1}{2}$).
 In some sense, the first transport plan is ``more optimal'' or ``less wasteful'' than the second one. 
\end{enumerate}
To make the last point precise it is convenient to make use of \emph{cost functions}. 

\paragraph{Cost functions, distances, and weights.}
In the context of transport plans, one is often interested in the \emph{cost} of a transport plan whenever the underlying space is equipped with a suitable cost function, for instance, a metric. This is part of the so-called theory of \emph{(optimal) transport} (see the first chapters of~\cite{villani} for an overview). 

A cost function on $X$ is a function $c:X\times X\to [0,\infty]$. For the purposes of the present paper, we can interpret the quantity $c(x,y)$ as the ``cost of going from $x$ to $y$''. We are in particular interested in cost functions such that $c(x,x)=0$, and satisfying the triangle inequality $c(x,z)\le c(x,y)+c(y,z)$. 
If $X$ is equipped with such a cost function, transport plans can be assigned a cost as well (see \Cref{catcouplings}). The cost of a transport plan between the distributions $p$ and $q$ can be interpreted as the cost of ``moving the mass from the pile $p$ to the pile $q$ according to the specified plan''. 
If the cost function $c$ satisfies a triangle inequality, the costs of transport plans satisfy one too, whenever two plans are composed. 
The resulting structure is called a \emph{weighted category}, or \emph{normed category} (see \Cref{terminology} for the terminology), a particular kind of enriched category. 

Weighted categories have been around at least since \cite{lawvere}, and have recently sparked some interest in different fields of mathematics, including geometry and topological data analysis \cite{interleaving,categorieswithnorms,schroderbernstein}. They can be considered a common generalization of categories and metric spaces.
In optimal transport theory, given probability measures $p$ and $q$ on a metric space, one can define their \emph{Wasserstein distance} by taking the infimum of the cost of transport plans between $p$ and $q$ (see \Cref{catcouplings} as well as \cite[Section~6]{villani}). This is an instance of a more general categorical construction: given objects $X$ and $Y$ of a weighted category, one can consider all the weights of the morphisms $X\to Y$ and take their infimum, the ``minimal cost of going from $X$ to $Y$'', and obtain this way a pseudo-quasimetric between the objects.
Several metric spaces appearing in mathematics can be obtained in this way: besides Wasserstein spaces, also the Gromov-Hausdorff distance is an example of this construction, as defined for example in \cite[Proposition~7.3.16 and Exercise~7.3.26]{burago}.  
The notion of weighted category is therefore very ``primitive'', it underlies several contemporary geometric constructions. 

As we said above, the main theme of this work is the idea of \emph{lifting}, which is common to both category theory and geometry, and therefore well-suited for treatment in terms of weighted categories. 
Indeed, lenses in category theory and submetries in metric geometry are very much related. We incorporate the two concepts by defining \emph{weighted lenses} in \Cref{lenses}, which are lenses where the liftings are weight-preserving.

The main result of this work shows that under some natural conditions, a metric lens between two spaces gives rise to a weighted lens between the respective categories of couplings (\Cref{main}), making mathematically precise the idea of ``lifting transport plans while preserving their cost''.

\paragraph{Outline.}
In \Cref{wcats} we review some of the basic concepts of weighted category theory.
The definitions of embedding and of symmetry for weighted categories seem to be new, but can be seen as special cases of the enriched-categorical notion.

In \Cref{lenses} we consider \emph{lenses}, the categorical structures formalizing the idea of ``lifting'' used in this work. We review two related notions of lens, for sets and for categories, and define weighted versions of them. We then show the link with submetries in metric geometry, which seems to be a new idea. 

In \Cref{catcouplings} we define, for each standard Borel space $X$, a category $PX$ whose objects are probability measures on $X$, whose morphisms are couplings, and where composition is given by means of conditional products. If $X$ is equipped with a cost function, we can turn $PX$ into a weighted category (in different ways, one for each integer, analogously to $L^p$ spaces), which after optimization give the famous Wasserstein spaces. 

In \Cref{mains} we state and prove the main statements of this work. 
\Cref{pushforward} and \Cref{embedding} show that pushforwards of probability measures are functorial, and that pushforwards along embeddings are embeddings. 
The most important result is \Cref{main}, which shows that a set-based lens between two spaces gives rise to a categorical lens between the respective spaces of probability measures, effectively lifting entire couplings along the chosen lifting of points. This gives in particular a functor between measurable lenses and weighted lenses (\Cref{funlensestolenses}). In \Cref{prodlens} we show that this includes the special case of product projections, where the lifting is given by taking the conditional product. 
All these results come with a weighted version in the metric and pseudo-quasimetric case. 

\paragraph{Acknowledgements.}

I would like to thank Tobias Fritz, Slava Matveev, and Sharwin Rezagholi for the fruitful conversations which eventually led to this work, and hopefully to many more. 
I would also like to thank Bryce Clarke and Matthew DiMeglio and David Spivak for the pointers on lenses, 
and Sam Staton and his whole group for the support and for the interesting discussions.
Finally, I thank the anonymous reviewers for their very helpful remarks and suggestions. 
 
\section{Weighted categories}\label{wcats}

Here we recall the main ideas of weighted categories and functors, their relationship with ordinary categories and with metric spaces, and their notion of symmetry, isomorphism, and embedding. 
Some of the definition given here are technically new (such as the notion of symmetry, \Cref{defsym}), but very similar to well known structures (such as dagger categories). 

\subsection{Main definitions}
 
\begin{deph}\label{defwcat}
 A \emph{weighted category} is a category $\cat{C}$ where to each arrow $f:X\to Y$ we assign a value $w(f)\in [0,\infty]$ called the \emph{weight} of $f$,
 such that
 \begin{itemize}
  \item all identities have weight zero;
  \item for every pair of composable arrows $f:X\to Y$ and $g:Y\to Z$, we have the ``triangle inequality''
  \begin{equation}\label{triangleinequality}
   w(g\circ f) \le w(g) + w(f) .
  \end{equation}
 \end{itemize}
\end{deph}
 
In this work we will mostly consider small categories, where the objects form a set. This will always be implicitly assumed from now on, unless otherwise stated.

\begin{remark}\label{terminology}
Some other authors use other names, such as ``normed categories''.
Also \cite{grandisweights} and \cite{interleaving} call these ``weighted categories''. In particular, in \cite{grandisweights} it is argued that these structures are more lax than what one may want from a norm, they are more similar to \emph{seminorms}. Probably for similar reasons, in \cite{schroderbernstein} these structures are called ``seminormed categories''. 
The original works~\cite{lawvere} and~\cite{categorienormate}, call them ``normed categories'' (the latter in Italian, ``categorie normate''), while
\cite{categorieswithnorms} and \cite{schroderbernstein} reserve that term for weighted categories which satisfy (different) special properties. (For the latter, see also our \Cref{isoopt}.)
We choose the name ``weighted categories'' as it seems to be the one giving rise to the least ambiguity. 
\end{remark}

Readers interested in enriched category theory~\cite{basicconcepts} can view weighted categories as enriched categories, where the base of the enrichment is the category of weighted sets~\cite[Section~1.4]{grandisweights}. We will not explicitly adopt this point of view in this work, except for some conventions. The enrichment approach is worked out in detail in~\cite{grandisweights}.  

Every ordinary category can be considered a weighted category where all arrows have weight zero.
Because of that, all the statements that we make for weighted categories also hold for ordinary categories. 
At the other end of the spectrum, weighted categories also generalize metric spaces, and pseudo-quasimetric spaces.
 
\begin{eg}\label{deflms}
 A \emph{pseudo-quasimetric space} (more briefly, \emph{pq-metric space}) or \emph{Lawvere metric space} is a set $X$ together with a ``cost'' function $c:X\times X\to [0,\infty]$, such that 
 \begin{itemize}
  \item for each $x\in X$, we have $c(x,x)=0$;
  \item for each $x$, $y$ and $z\in X$, we have the triangle inequality
  $$
  c(x,z) \le c(x,y) + c(y,z) .
  $$
 \end{itemize}
 
 A pq-metric space can be seen as a weighted category with exactly one arrow between any two objects $x$ and $y$, and the quantity $c(x,y)$ is the weight of the arrow.
\end{eg}

 In analogy to metric spaces, we will sometimes call $c$ \emph{distance}. Note that in general, for a pq-metric,
\begin{itemize}
 \item if $c(x,y)=0$, it is not necessarily the case that $x=y$;
 \item symmetry is not required: $c(x,y)$ can be different from $c(y,x)$. In particular, one of the two may be zero when the other one is not;
 \item infinities are allowed.
\end{itemize}

Here is another example, again from metric geometry.

\begin{eg}\label{lipschitzcat}
 Consider the category $\cat{Lip}$ where
 \begin{itemize}
  \item objects are metric spaces;
  \item morphisms are (all) functions;
  \item the weight of a morphism $f:X\to Y$ is the maximum between $0$ and the \emph{logarithm} of its Lipschitz constant,
  $$
  \sup_{x,x'\in X} \log \dfrac{d\big(f(x),f(x')\big)}{d(x,x')} ,
  $$
  which is possibly infinite (and with the convention that $0/0=1$). 
 \end{itemize}
 $\cat{Lip}$ is a (large) weighted category. (One can also define \emph{multiplicative weights} and avoid taking the logarithm, see \cite{grandisweights} for more on this.), as the formula \eqref{lenscond} will show. 
\end{eg}

Every weighted category gives a pq-metric space by ``optimization over the arrows'', as follows. 

\begin{deph}
 Let $\cat{C}$ be a weighted category. The \emph{optimization} of $\cat{C}$ is the pq-metric space $\Opt(\cat{C})$ where the points are the objects of $\cat{C}$, and the distance is given by 
 \begin{equation}\label{infmetric}
  c(X,Y) \coloneqq \inf_{f:X\to Y} w(f) ,
 \end{equation} 
 where by convention, if there are no arrows from $X$ to $Y$, the infimum over the empty set gives $\infty$.
 As it can be easily checked, this quantity is indeed a pq-metric.
\end{deph}

This construction was ``left as exercise'' in \cite[Introduction, pg.~139-140]{lawvere}, and fully worked out (and generalized) in \cite{categorienormate}. It is an instance of the ``change of enrichment'' construction of enriched category theory.

Note that in general the minimum is not realized.
 
\begin{deph}\label{optcomplete}
 A weighted category $\cat{C}$ is called \emph{optimization-complete} if for every two objects $X$ and $Y$ of $\cat{C}$, there exists a morphism $X\to Y$ of minimal weight.
 That is, the infimum in \eqref{infmetric} is attained. 
\end{deph}

As functors between weighted categories it is useful to take those functors which are ``weight-nonincreasing'' (or ``1-Lipschitz''):

\begin{deph}\label{defwfun}
 Let $\cat{C}$ and $\cat{D}$ be weighted categories. A \emph{weighted functor} $F:\cat{C}\to\cat{D}$ is a functor such that for each morphism $f:X\to Y$ of $\cat{C}$ 
 $$
 w(Ff) \le w(f) .
 $$
 We denote by $\cat{WCat}$ the category of (small) weighted categories and weighted functors.
\end{deph}

If $\cat{C}$ and $\cat{D}$ are ordinary categories, and we consider them as weighted categories where all the weights are zero, a weighted functor $\cat{C}\to\cat{D}$ is the same thing as an ordinary functor. Categorically, this gives a fully faithful functor $\cat{Cat}\to\cat{WCat}$.

If $X$ and $Y$ are pq-metric spaces and we consider them as weighted categories, a function $f:X\to Y$ is a weighted functor if and only if it is 1-Lipschitz. (Since arrows are unique, identities and composition are trivially preserved.) Calling $\cat{pqMet}$ the category of pq-metric spaces and 1-Lipschitz maps, we then have a fully faithful functor $\cat{pqMet}\to\cat{WCat}$.
Somewhat conversely, any weighted functor between weighted categories gives a 1-Lipschitz function between the pq-metric spaces obtained via optimization. Therefore, $\Opt$ is a functor $\cat{WCat}\to\cat{pqMet}$.

The following definition can be seen as the many-arrow analogue of the symmetry of a metric, or as a weighted analogue of dagger categories.
Similarly to the unweighted case, denote by $\cat{C}^\op$ the category obtained by reversing all arrows of $\cat{C}$, with the same weights.

\begin{deph}\label{defsym}
 A \emph{symmetric weighted category} is a weighted category $\cat{C}$ equipped with a weighted functor $\dagger:\cat{C}^\op\to\cat{C}$, called the \emph{symmetry} or \emph{dagger}, such that 
 \begin{itemize}
  \item for each object $X$, $X^\dagger=X$ (i.e.~$\dagger$ is the identity on the objects);
  \item for each morphism $f$, we have that $f^{\dagger\dagger}=f$ (i.e.~$\dagger$ is an involution, $\dagger^2=\id$).
 \end{itemize}
\end{deph}

In other words, in a symmetric weighted category to each morphism $f:X\to Y$ there is a morphism in the opposite direction, $f^\dagger:Y\to X$, necessarily of the same weight as $f$. (And moreover, this correspondence preserves identities and composition).

For example, if we take a pq-metric space and consider it as a weighted category as in \Cref{deflms}, it is symmetric as a weighted category if and only if the cost function is symmetric (for example, if it is a metric). 
Similarly, if $\cat{C}$ is a symmetric weighted category, then the cost function of the optimization $\Opt(\cat{C})$ will be symmetric. 

\subsection{Isomorphisms}

Of particular interest is the notion of isomorphism in a weighted category, which is again bringing together category-theoretic and geometric concepts.

\begin{deph}\label{defiso}
 Let $X$ and $Y$ be objects in a weighted category $\cat{C}$. A \emph{quasi-isomorphism} between $X$ and $Y$ is a pair of morphisms $f:X\to Y$ and $g:Y\to X$ of finite weight, such that $g\circ f=\id_X$ and $f\circ g=\id_Y$.
 
 An \emph{isomorphism} between $X$ and $Y$ is a quasi-isomorphism pair $(f,g)$ such that both $f$ and $g$ have weight zero.
 
 We also write $g=f^{-1}$, $f=g^{-1}$, and sometimes we denote the pair $(f,g)$ simply by $f$. If an isomorphism (resp.~quasi-isomorphism) between $X$ and $Y$ exists, we say that $X$ and $Y$ are \emph{isomorphic} (resp.~\emph{quasi-isomorphic}), and we write $X\cong Y$ (resp.~$X\simeq Y$).
\end{deph}

Our term ``quasi-isomorphism'' is unrelated to the notion in homological algebra with the same name. 

\begin{remark}
Just as for the definition of weighted categories, different authors refer to weighted isomorphisms in different ways. 
In general, only invertible morphisms \emph{of weight zero} tend to behave like ``full'' isomorphisms. (See the next two examples.) 
This intuition can be made mathematically precise: considering weighted categories as enriched categories, the invertible morphisms of weight zero are exactly the isomorphisms in the sense of enriched category theory (as defined for example in~\cite{basicconcepts}). 
In \cite{schroderbernstein} these morphisms are called \emph{norm isomorphisms}, in \cite{categorieswithnorms} they are called \emph{$0$-isomorphisms}.
\end{remark}

\begin{eg}
If $X$ is a pq-metric space, considered as a weighted category, the points $x,y\in X$ are isomorphic if and only if $c(x,y)=0$ and $c(y,x)=0$. In particular, in a metric space, $x\cong y$ of and only if $x=y$. The points $x$ and $y$ are quasi-isomorphic if and only if they have finite distance in both directions.
\end{eg}

\begin{eg}
The isomorphisms in the category $\cat{Lip}$ of \Cref{lipschitzcat} are precisely the isometries. 
The quasi-isomorphisms are precisely the bi-Lipschitz maps. 
\end{eg}

\begin{remark}\label{isoopt}
Given isomorphic objects $X$ and $Y$ in a weighted category $\cat{C}$, they will also be isomorphic in $\Opt(\cat{C})$, i.e.~have distance zero in both directions. 
Instead, if $X$ and $Y$ are isomorphic in $\Opt(\cat{C})$, they are not necessarily isomorphic in $\cat{C}$: all we can say is that for every $\e > 0$, there are morphisms $f_\e:X\to Y$ and $g_\e:Y\to X$ of weight less than $\e$. If $\cat{C}$ is optimization-complete, then there are also morphisms $f:X\to Y$ and $g:Y\to X$ of weight zero realizing the two minima. However, this still does not guarantee that $f$ and $g$ are inverses. 
Similar consideration can be given for quasi-isomorphisms.

The authors of \cite{schroderbernstein} define a \emph{normed category} as a weighted category $\cat{C}$ where, in our terminology, for every two objects $X$ and $Y$ we have that $X\cong Y$ in $\cat{C}$ if and only if $X\cong Y$ in $\Opt(\cat{C})$. 
(Recall however that the term ``normed category'' is used differently by different authors, see \Cref{terminology}.)
\end{remark}

\subsection{Embeddings}

Embeddings of weighted categories are the common generalization of isometric embeddings of metric spaces and of fully faithful functors. 

\begin{deph}
 Let $\cat{C}$ and $\cat{D}$ be weighted categories. An \emph{embedding} of $\cat{C}$ into $\cat{D}$ is a weighted functor $F:\cat{C}\to\cat{D}$ such that for every pair of objects $X$ and $Y$ of $\cat{C}$, the function 
 $$
 \begin{tikzcd}[row sep=0]
  \Hom_\cat{C}(X,Y) \ar{r} & \Hom_\cat{D}(FX,FY) \\
  f \ar[mapsto]{r} & Ff
 \end{tikzcd}
 $$
 is a weight-preserving bijection.
\end{deph}

\begin{eg}
 If $X$ and $Y$ are metric spaces (or even pq-metric spaces), considered as weighted categories, a map $f:X\to Y$ is an embedding in the sense above exactly if it is an isometric embedding.
 Somewhat conversely, an embedding of weighted categories $F:\cat{C}\to\cat{D}$ induces an isometric embedding of the pq-metric spaces $\Opt(\cat{C})\to\Opt(\cat{D})$. 
\end{eg}

In ordinary category theory, fully faithful functors are injective up to isomorphism, and in metric geometry, isometric embeddings are injective up to zero distances. 
Both these statements can be considered a special case of the following.

\begin{prop}
 Let $F:\cat{C}\to\cat{D}$ be an embedding of weighted categories, and let $X$ and $Y$ be objects of $\cat{C}$. If $FX$ and $FY$ are isomorphic (resp.~quasi-isomorphic) in $\cat{D}$, then $X$ and $Y$ are already isomorphic (resp.~quasi-isomorphic) in $\cat{C}$.
\end{prop}

\begin{proof}
 Let $f:FX\to FY$ be an isomorphism of $\cat{D}$, with inverse $g:FY\to FX$, both of weight zero. Since $F$ is an embedding, it is surjective on morphisms, and so necessarily $f=F\tilde{f}$ and $g=F\tilde{g}$ for some $\tilde{f}:X\to Y$ and $\tilde{g}:Y\to Z$, also of weight zero. Now, 
 $$
 F(\tilde{g}\circ\tilde{f}) = F\tilde{g}\circ F\tilde{f} = g\circ f = \id_{FX} = F(\id_X),
 $$
 but since $F$ is injective on morphisms, we have $\tilde{g}\circ\tilde{f}=\id_X$. Similarly, $\tilde{f}\circ\tilde{g}=\id_Y$.
 Therefore $\tilde{f}$ and $\tilde{g}$ form an isomorphism pair between $X$ and $Y$. 
 
 The quasi-isomorphism case works analogously.
\end{proof}

\section{Lenses and submetries}\label{lenses}

In the setting of weighted categories, we can bring together the concept of \emph{lens} (from category theory) and the concept of \emph{submetry} (from metric geometry).

When it comes to lenses, different authors mean different specific concepts under the same name. 
Here we will use two (related) flavors of lenses: \emph{set-based} lenses and \emph{delta} lenses, and we will provide weighted versions of both.

In our context, we can interpret lenses, as well as submetries, in terms of ``lifting''. 
Indeed, suppose that we have two sets $X$ and $Y$, and a ``projection'' of $X$ onto $Y$. We sketch the situation in the following picture.
 
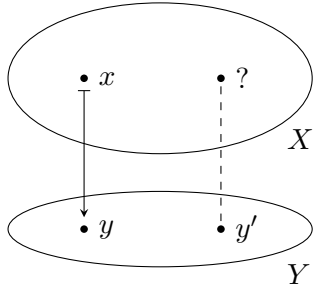
\begin{center}
 \begin{tikzpicture}[%
    baseline=(current  bounding  box.center),
    >=stealth,
    bullet/.style={
       fill=black,
       circle,
       minimum width=1pt,
       inner sep=1pt
     }]
   \draw (0,2) ellipse (2cm and 1cm);
   \node[label=right:$X$] (x) at (1.4,1.2) {};
   
   \draw (0,0) ellipse (2cm and 0.5cm);
   \node[label=right:$Y$] (x) at (1.4,-0.6) {};
   
   \node[bullet,label=right:$x$] (x) at (-1,2) {};
   \node[bullet,label=right:$?$] (x') at (0.8,2) {};
   \node[bullet,label=right:$y$] (y) at (-1,0) {};
   \node[bullet,label=right:$y'$] (y') at (0.8,0) {};
   
   \draw[-,dashed, shorten <=0.5mm, shorten >=0.5mm] (x') -- (y');
   \draw[|->, shorten <=1mm, shorten >=1mm] (x) -- (y);
 \end{tikzpicture}
 \end{center} 
 Given a point $x$ of $X$, as in the picture above, we can project it down to a point $y$ of $Y$. If we now move from $y$ to a different point $y'$ in $Y$, can we lift $y'$ back to a point of $X$ which ``corresponds to $x$ in a canonical way''? 
 Here, ``canonical'' may mean different things in different contexts. For metric spaces we may want, for example, that $d(x,x')=d(y,y')$. In other situations, we may want a form of parallel transport instead.
 Lenses and submetries make this intuition precise in these different settings.
 In particular, \emph{set-based lenses} can be seen as lifting \emph{points}, while \emph{delta lenses} can be seen as lifting \emph{paths}, in a way that preserves identities and composition.

\subsection{Lifting points}
 
Here we use the notion of (set-based) lens first defined in \cite{lens-first}, and analyzed in detail for example in \cite{setbased} (where they are called ``state-based lenses'').
 
\begin{deph}
 A \emph{set-based lens} consists of 
\begin{itemize}
 \item two sets $X$ and $Y$;
 \item a function $f:X\to Y$;
 \item a function $\varphi:X\times Y \to X$,
\end{itemize}
such that the following equations are satisfied for every $x\in X$ and $y\in Y$,
\begin{equation}\label{lenslawseq}
 f\big( \varphi(x,y) \big) = y 
 \qquad\qquad
 \varphi\big( x, f(x) \big) = x 
 \qquad\qquad
 \varphi\big( \varphi(x,y), y' \big) = \varphi(x, y') 
\end{equation}
or equivalently, such that the following diagrams commute,
\begin{equation}\label{lenslawsdiag}
\begin{tikzcd}
 X\times Y \ar{r}{\varphi} \ar{dr}[swap]{\pi_2} & X \ar{d}{f} \\
 & Y
\end{tikzcd}
\qquad
\begin{tikzcd}
 X \ar{r}{(\id,f)} \ar{dr}[swap]{\id} & X\times Y \ar{d}{\varphi} \\
 & X
\end{tikzcd}
\qquad
\begin{tikzcd}
 X\times Y\times Y \ar{r}{\varphi\times\id} \ar{d}{\pi_{1,3}} & X\times Y \ar{d}{\varphi} \\
 X\times Y \ar{r}{\varphi} & X
\end{tikzcd}
\end{equation}
where the maps $\pi$ are the product projections onto the respective factors.
\end{deph}

We will refer to the conditions \eqref{lenslawseq} (or equivalently, \eqref{lenslawsdiag}) as the \emph{lens laws}. In particular we call them, in order, the \emph{lifting law}, the \emph{identity law}, and the \emph{composition law}. We can interpret the maps and the laws as follows.
\begin{itemize}
 \item The map $f$ is a projection from $X$ to $Y$.
 \item The map $\varphi$ takes a point $x\in X$ and a point $y\in Y$, and gives a lifting of $y$ back to a point $\varphi(x,y)\in X$ which in some sense ``corresponds to $x$, but above $y$''.
 \item The lifting law says that indeed, the chosen lifting $\varphi(x,y)$ is in the preimage $f^{-1}(y)$, i.e.~``lies above $y$''.
 \item The identity law says that the point over $f(x)$ corresponding to $x$ is $x$ itself.
 \item The composition law says that the procedure can be iterated, and the resulting lift does not change.
\end{itemize}

Lenses can be composed. Indeed, if $(f,\varphi)$ is a lens between $X$ and $Y$ and $(g,\psi)$ is a lens between $Y$ and $Z$, the composite lens between $X$ and $Z$ is given by 
$$
\begin{tikzcd}
 X \ar{r}{f} & Y \ar{r}{g} & Z
\end{tikzcd}
$$
and
$$
\begin{tikzcd}[sep=large]
 X\times Z \ar{r}{(\id,f)\times\id} & X\times Y\times Z \ar{r}{\id\times\psi} & X\times Y \ar{r}{\varphi} & X .
\end{tikzcd}
$$
This can be interpreted in terms of sequential liftings: given $x\in X$ and $z\in Z$, we first lift $z$ to $Y$ using $\psi$, and then lift the result to $X$ using $\varphi$.

\begin{remark}\label{prodproj}
 Set-based lenses can always be turned into product projections, making the set $X$ isomorphic to a product $Y\times Z$ \cite[Example~6.74]{poly-book}.
 To see it, take for example $Z$ to be the set of sections $s:Y\to X$ of $f$ (such that $f\circ s= \id_Y$) with the ``rigidity'' property that $\varphi(s(y),y')=s(y')$ for all $y,y'\in Y$.
 This way, the map $f:X\to Y$ corresponds to the product projection
 \begin{equation}\label{fprod}
 \begin{tikzcd}[row sep=0] 
  Y\times Z \ar{r}{f} & Y \\
  (y,z) \ar[mapsto]{r} & y
 \end{tikzcd}
 \end{equation}
 and the map $\phi$ is then the canonical lifting
 \begin{equation}\label{phiprod}
 \begin{tikzcd}[row sep=0]
  (Y\times Z)\times Y \ar{r}{\varphi} & Y\times Z \\
  \big((y,z),y'\big) \ar[mapsto]{r} & (y',z) .
 \end{tikzcd}
 \end{equation}
 One can interpret this as ``choosing the point in $Y\times Z$ which lies over $y'$, but which is at the same height as $(y,z)$''. Geometrically, $\varphi$ can be seen as a choice of the ``horizontal'' direction in the product $Y\times Z$, analogous to a choice of connection on a (trivial) bundle.
\end{remark}

The definition of set-based lens can be given in any category with products~\cite{algebras-update-strategies}. For example, we are interested in \emph{measurable} lenses.

\begin{deph}
 A \emph{measurable lens} consists of a lens $(f,\varphi)$ between measurable spaces $X$ and $Y$, such that the functions $f$ and $\varphi$ are measurable functions.
\end{deph}

Given a measurable lens $(f,\varphi)$ between measurable spaces $X$ and $Y$, by \Cref{prodproj} it is still true that on the underlying sets it can be seen as a product projection $Y\times Z\to Y$. However, in principle, the $\sigma$-algebra on $X$ might be different than the product $\sigma$-algebra of $Y\times Z$.
The question of whether measurable lenses can be always turned into product projections in the category of measurable spaces is still open, and we leave it to future work. 
(See however \Cref{prodprojmet}.)

Let's now move to the metric setting, and let's recall the notion of submetry. Traditionally, it is given as follows~\cite{submetries,weaksubmetries}.\footnote{In the Riemannian geometry literature, sometimes people use a definition similar to this one, but where the condition holds only for $r\le R$ for some fixed bound $R$. We will not use such notion here.}
\begin{deph}\label{defsubmetryalt}
 Let $X$ and $Y$ be metric spaces. Denote by $B_X(x,r)$ the closed ball in $X$ of center $x$ and radius $r$, and define $B_Y$ analogously. 
 A map $f:X\to Y$ is a \emph{submetry} if for every $x\in X$ and for every $r>0$ we have that 
 $$
 f \big( B_X(x,r) \big) = B_Y \big( f(x), r \big) .
 $$
\end{deph}
As it can be checked, this definition can be rephrased in the following equivalent way, which is closer to our idea of ``lifting''.
\begin{deph}[alternative]\label{defsubmetry}
 Let $X$ and $Y$ be metric spaces. A function $f:X\to Y$ is a \emph{submetry} if it is 1-Lipschitz, and for every $x\in X$ and $y'\in Y$ there exists $x'\in X$ such that 
 $$
 f(x')=y' \qquad\mbox{and}\qquad  d(x,x') = d\big(f(x), y'\big) .
 $$
\end{deph}

This definition allows us to draw a link between submetries and lenses.\footnote{Callum Reader, in personal communication, pointed out to me that there is an additional, very deep connection between submetries and lenses. Indeed, one can define \emph{enriched} lenses between enriched categories~\cite{enriched-lenses}, and if one instantiates the definition in the case of Lawvere metric spaces, one obtains exactly \emph{weak submetries}, which are defined by replacing closed balls with open balls in \Cref{defsubmetryalt}, see~\cite[Section~1]{weaksubmetries}.}

\begin{deph}\label{metriclens}
 A \emph{metric lens} is a lens $(X,Y,f,\varphi)$ where $X$ and $Y$ are metric spaces, $f$ is 1-Lipschitz, and $f$ and $\varphi$ satisfy the requirement that 
\begin{equation}\label{conddist}
d\big( x, \varphi(x,y) \big) = d\big( f(x), y \big) .
\end{equation}
\end{deph}

Note that this is not the internal notion of lens in a category of metric spaces: $\varphi$ is not required to be 1-Lipschitz, or even continuous. (It is rather a special case of weighted lens, more on that later.)

Every metric lens is canonically a submetry. Indeed, let $(X,Y,f,\varphi)$ be a metric lens, and let $x\in X$ and $y'\in Y$. A point $x'\in X$ such that 
 $$ 
  f(x')=y' \qquad\mbox{and}\qquad  d(x,x') = d\big(f(x), y'\big) .
 $$
is simply given by $x'\coloneqq \varphi(x,y')$. 
Here is a partial converse to this fact. 

\begin{prop}\label{partialconverse}
 Let $f:X\to Y$ be a submetry. For every $x\in X$ and $y'\in Y$, \emph{choose} an $x'\in X$ such that 
 $$
 f(x')=y' \qquad\mbox{and}\qquad  d(x,x') = d\big(f(x), y'\big) .
 $$
 Denote this chosen $x'$ by $\varphi(x,y')$, so that we have a function $\varphi:X\times Y\to Y$. 
 
 Then the pair $(f,\varphi)$ satisfies the lifting and identity axioms of lenses \eqref{lenslawseq}. 
\end{prop}

\begin{proof}
 The lifting law says that 
 $$
 f\big( \varphi(x,y) \big) = y,
 $$
 which is just the condition that $f(x')=y'$.
  
 The identity law says that 
 $$
 \varphi\big( x, f(x) \big) = x.
 $$
 To prove it, note that by the condition on distances,
 $$
 d \big( x, \varphi\big( x, f(x) \big) \big) = d \big( f(x), f(x) ) = 0,
 $$
 therefore $x= \varphi\big( x, f(x) \big)$.
\end{proof}

Therefore, one can view a metric lens as ``a submetry with a choice of liftings, compatible with composition''. 

One can generalize the notion of metric lens to arbitrary pq-metric spaces. (Note that for pq-metric spaces a ``submetry'' only satisfies the identity law up to distance zero, i.e.~$\varphi\big( x, f(x) \big)$ and $x$ are not necessarily the same point, but have distance zero.) 

\begin{remark}\label{prodprojmet}
 Differently from the set case, metric lenses are not necessarily just product projections. Once again, they are product projections on the underlying sets, but the set $X\cong Y\times Z$ may be equipped with a different metric than the product one. In particular, we get a metric lens as long as we equip $Y\times Z$ with a metric such that for every $z\in Z$,
 $$
 d\big( (y,z), (y',z) \big) = d(y,y') .
 $$
 Moreover, the map $\varphi$ is 1-Lipschitz whenever the metric is larger or equal than the product metric, i.e.
 $$
 d\big( (y,z), (y',z') \big) \ge \max\{ d(y,y'), d(z,z') \} .
 $$
 These two conditions are very natural conditions for a metric on $Y\times Z$, and most metrics on a product used in the literature do satisfy these requirements.
\end{remark}

\subsection{Lifting paths}

We now turn to categories and weighted categories. This time we don't just want to lift points, but also paths, or arrows. 
We will consider a weighted analogue of the structure called \emph{delta lens}, first defined in \cite{delta-lens-first}. 
(See also \cite{bryce} for the internal version of delta lenses.)

We recall the original definition, and then present the weighted version.

\begin{deph}\label{lifting}
 Let $\cat{E}$ and $\cat{B}$ be categories. Let $F:\cat{E}\to\cat{B}$ be a functor, let $b:B\to B'$ be a morphism of $\cat{B}$, and
 let $E$ be an object of $\cat{E}$ with $F(E)=B$. 
 A \emph{lifting of $b$ at $E$} is an object $E'$ of $\cat{E}$ such that $F(E')=B'$, together with an arrow $\Phi(E,b):E\to E'$ of $\cat{E}$ such that $F\big(\Phi(E,b)\big)=b$.
\end{deph}

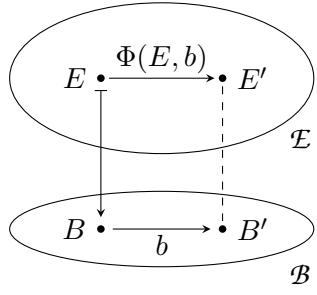
 \begin{center}
 \begin{tikzpicture}[%
    baseline=(current  bounding  box.center),
    >=stealth,
    bullet/.style={
       fill=black,
       circle,
       minimum width=1pt,
       inner sep=1pt
     }]
   \draw (0,2) ellipse (2cm and 1cm);
   \node[label=right:$\cat{E}$] (E) at (1.4,1.2) {};
   
   \draw (0,0) ellipse (2cm and 0.5cm);
   \node[label=right:$\cat{B}$] (B) at (1.4,-0.6) {};
   
   \node[bullet,label=left:$E$] (x) at (-0.8,2) {};
   \node[bullet,label=right:$E'$] (x') at (0.8,2) {};
   \node[bullet,label=left:$B$] (y) at (-0.8,0) {};
   \node[bullet,label=right:$B'$] (y') at (0.8,0) {};
   
   \draw[-,dashed, shorten <=0.5mm, shorten >=0.5mm] (x') -- (y');
   \draw[|->, shorten <=1mm, shorten >=1mm] (x) -- (y);
   \draw[->, shorten <=0.5mm, shorten >=0.5mm] (x) -- (x');
   \draw[->, shorten <=1mm, shorten >=1mm] (y) -- (y');
   
   \node[label=below:$b$] (bm) at (0,0.2) {};
   \node[label=above:$\Phi{(E,b)}$] (phi) at (0,1.8) {};
 \end{tikzpicture}
 \end{center} 

\begin{deph}\label{lens}
 Let $\cat{E}$ and $\cat{B}$ be categories. 
 A \emph{delta lens} or \emph{categorical lens} from $\cat{E}$ to $\cat{B}$ consists of the following data.
 \begin{itemize}
  \item A functor $F:\cat{E}\to\cat{B}$.
  \item For each morphism $b:B\to B'$ of $\cat{B}$ and each object $E$ of $\cat{E}$ with $F(E)=B$, a chosen lifting $\Phi(E,b):E\to E'$ of $b$ satisfying the following conditions.
  \begin{itemize}
  \item The identities are lifted to identities, i.e.~$\Phi(E,\id_B)=\id_{E}$.
  \item The choice of liftings preserves composition. That is, given $b:B\to B'$ and $b':B'\to B''$, consider the chosen lifting $\Phi(E,b):E\to E'$ of $b$, and then, at its codomain $E'$ consider the lifting $\Phi(E',b')$ of $b'$. We require that $\Phi(E',b')\circ \Phi(E,b)=\Phi(E,b'\circ b)$.
  \end{itemize}
 \end{itemize}
\end{deph}

A famous particular case of lens is given by \emph{split Grothendieck opfibrations}, where liftings are moreover \emph{initial} in some sense. 
A somewhat trivial case is given by set-based lenses, which are equivalently delta lenses between \emph{codiscrete} categories (a unique arrow between any two objects).
See for example \cite{bryce} for more on this.

Here is the weighted analogue of liftings and lenses.\footnote{Once again, this is an instance of \emph{enriched lens} as given in \cite{enriched-lenses}, for the case of weighted categories.}
\begin{deph}
 Let $\cat{E}$ and $\cat{B}$ be weighted categories. Let $F:\cat{E}\to\cat{B}$ be a weighted functor, let $b:B\to B'$ be a morphism of $\cat{B}$, and
 let $E$ be an object of $\cat{E}$ with $F(E)=B$. 
 A \emph{weight-preserving lifting of $b$ at $E$} is a lifting $\Phi(E,b)$ of $b$ such that $w\big(\Phi(E,b)\big) = w(b)$.
 
A \emph{weighted lens} is defined similarly to a delta lens, but where all the liftings are required to be weight-preserving.
\end{deph}

Consider now pq-metric spaces $X$ and $Y$ as weighted categories as in \Cref{deflms}. A weighted lens between them consists of  
\begin{itemize}
 \item a 1-Lipschitz map $f:X\to Y$;
 \item for every $x\in X$ and $y'\in Y$, a chosen lifting of $y'$ to a point $x'\in f^{-1}(y')$ such that $d(f(x),y')=d(x,x')$. 
\end{itemize}
This is, in other words, a metric lens in the sense of \Cref{metriclens}.

Also delta lenses can be composed, analogously to set-based lenses. For the details, see \cite{bryce}. 
For convenience, let's define the following categories.
\begin{itemize}
 \item $\cat{SetLens}$ is the category of sets and set-based lenses.
 \item $\cat{pqMetLens}$ is the category of pq-metric spaces and metric lenses.
 \item $\cat{CatLens}$ is the category of (small) categories and delta lenses.
 \item $\cat{WCatLens}$ is the category of (small) weighted categories and weighted lenses. 
\end{itemize}
Between these categories we have the following fully faithful functors.
\begin{itemize}
 \item By considering two sets as a codiscrete categories, set-based lenses between them are the same as delta lenses. This gives a fully faithful functor $\cat{SetLens}\to\cat{CatLens}$,
 \item Analogously, by considering pq-metric spaces as weighted categories with a single arrow between any two objects, we get a fully faithful functor $\cat{pqMetLens}\to\cat{WCatLens}$. 
 \item Considering ordinary categories as having weight zero for all morphisms gives a fully faithful functor $\cat{CatLens}\to\cat{WCatLens}$. 
 \item One can construct a similar functor $\cat{SetLens}\to\cat{pqMetLens}$, making the following diagram commute.
 $$
\begin{tikzcd}
 \cat{SetLens} \ar{d} \ar{r} & \cat{pqMetLens} \ar{d} \\
 \cat{CatLens} \ar{r} & \cat{WCatLens}
\end{tikzcd}
$$
\end{itemize}

\section{The category of couplings}\label{catcouplings}

Here we briefly review some basic notions of optimal transport theory, mostly to fix the notation. 
For most of the constructions in this section we require the measurable spaces to be standard Borel, also known as ``Polish'': those measurable spaces whose $\sigma$-algebra can be obtained as the Borel $\sigma$-algebra of a complete, separable metric space. This is needed in particular to have conditional distributions, more on that later in this section. 
 
For $X$ a (generic) measurable space, denote by $\Sigma(X)$ its $\sigma$-algebra, and by $PX$ the set of probability measures over it. 
If $X$ and $Y$ are measurable spaces, denote by $X\times Y$ their Cartesian product, equipped with the product $\sigma$-algebra. 

\begin{deph}
 Let $X$ and $Y$ be measurable spaces. Let $p\in PX$ and $q\in PY$ be probability measures. A \emph{coupling} of $p$ and $q$ is a probability measure $s\in P(X\times Y)$ such that its marginals on $X$ and $Y$ are $p$ and $q$, respectively.
 
 We denote by $\Gamma(p,q)$ the space of couplings of $p$ and $q$.
\end{deph}

Given a coupling $s$ of $p$ and $q$, and if $A\subseteq X$ and $B\subseteq Y$ are measurable subsets, we will write $s(A\times B)$ as shorthand for the measure of the ``cylinder''
$$
s \Big( \big(\pi_1^{-1}(A)\times Y\big) \intersection \big(X\times \pi_2^{-1}(B)\big) \Big) ,
$$
where $\pi_1:X\times Y\to X$ and $\pi_2:X\times Y\to Y$ are the product projections. 

\begin{deph}
 Let $X$ be a measurable space, and let $c:X\times X\to[0,\infty]$ be a pq-metric. Let $p$ and $q$ be Borel probability measures on $X$, and let $s$ be a coupling of $p$ and $q$. For an integer $k\ge 1$, the \emph{$k$-cost} of the coupling $s$ is the quantity
 $$
 \cost_k(s) \coloneqq \sqrt[k]{ \int_{X\times X} c(x,y)^k \, s(dx\,dy) }.
 $$
\end{deph}

Given a standard Borel space $X$, and fixing $k$, we now construct a category whose objects are probability measures on $X$, and whose morphisms are couplings with the weight given by the $k$-cost.

Given $p\in PX$, the identity coupling $I_p$ is given, intuitively, by ``a copy of $p$ supported on the diagonal''. We can express this rigorously as
$$
I_p(A \times A') \coloneqq p(A\intersection A') .
$$

Therefore, the cost of the identity is zero for every $k$:
$$
\cost_k(I_p)^k = \int_{X\times X} c(x,y)^k\,I_p(dx\,dy) = \int_X c(x,x)^k\,p(dx) = 0.
$$

In order to define composites, we have to use conditional probabilities.

\subsection{Composition}

Let $X$ and $Y$ be measurable spaces. A \emph{Markov kernel} from $X$ to $Y$ is a mapping 
$$
\begin{tikzcd}[row sep=0]
X\times \Sigma(Y) \ar{r}{h} & {[0,1]} \\
(x,B) \ar[mapsto]{r} & h(B|x)
\end{tikzcd}
$$ such that, for each $x\in X$, $h(-|x):\Sigma(Y)\to[0,1]$ is a probability measure, and for each $B\in \Sigma(Y)$, $h(B|-):X\to [0,1]$ is a measurable function.

Markov kernels are often used to denote conditional probabilities, and this works particularly well in standard Borel spaces. We use the following standard ``conditioning'' result, which will underlie most of the constructions in the rest of this work. (For a reference, see for example~\cite[Section~452]{fremlin} and \cite[Section~10.4]{bogachev}.)

\begin{thm}\label{conditionals}
 Let $X$ and $Y$ be standard Borel spaces, let $p$ and $q$ be probability measures on $X$ and $Y$ respectively, and let $s$ be a coupling of $p$ and $q$. There exists a Markov kernel $\vec{s}$ from $X$ to $Y$ such that for every pair measurable subsets $A\subseteq X$ and $B\subseteq Y$,
  $$
  \int_{A} \vec{s}(B|x) \, p(dx) =  s(A\times B) ;
  $$
 Moreover, such a kernel $\vec{s}$ is unique $p$-almost surely.
\end{thm}
We call $\vec{s}$ the \emph{conditional distribution} of $Y$ given $X$ (associated to the coupling $s$). 

One can also condition $s$ on $Y$, obtaining a kernel from $Y$ to $X$. We denote this by $\cev{s}(A|y)$. We then have a form of Bayes' theorem, saying that for $A\subseteq X$ and $B\subseteq Y$ measurable,
\begin{equation}\label{bayes}
\int_{A} \vec{s}(B|x) \, p(dx) = s(A\times B) = \int_B \cev{s}(A|y)\, q(dy) .
\end{equation}

A conditional distribution of the identity coupling is given by the indicator function:
$$
\vec{I}_p (A|x) = \cev{I}_p (A|x)= 1_A(x) = \begin{cases}
                            1 & x\in A \\
                            0 & x\notin A .
                           \end{cases}
$$
or equivalently by the map assigning to $x\in X$ the Dirac delta measure $\delta_x$.

Let's now turn to the composition of couplings. This is a standard construction, it is used in optimal transport, under the name ``gluing'' \cite[Chapter~1]{villani}, as well as in statistics, where it's related to the ``conditional product'' construction \cite{cprod}, and in categorical probability (see \cite[Example~6.2]{simpson} and the references therein, as well as the more general construction of \cite[Definition~12.8]{markov}).

\begin{deph}
 Let $p$, $q$ and $r$ be probability measures on the standard Borel spaces $X$, $Y$ and $Z$, respectively. Let $s$ be a coupling of $p$ and $q$, and let $t$ be a coupling of $q$ and $r$. 
 The distribution $t\circ s$ on $X\times Z$ is defined by
 $$
 (t\circ s) (A\times C) \coloneqq  \int_{Y} \cev{s}(A|y) \, \vec{t}(C|y) \, q(dy)
 $$
 for all measurable subsets $A\subseteq X$ and $C\subseteq Z$. 
\end{deph}

We have that the marginal of $t\circ s$ on $X$, evaluated on the measurable subset $A\subseteq X$, is 
$$
 (t\circ s) (A\times Z) =  \int_{Y} \cev{s}(A|y) \cdot 1 \, q(dy) = s(A\times Y) = p(A) ,
$$
and the same can be done for $Z$, so that $t\circ s$ is a coupling of $p$ and $r$. 
Moreover, $t\circ s$ does not depend on the choice of the conditional distributions $\cev{s}$ and $\vec{t}$, since they are unique $q$-almost everywhere. 

The composite coupling gives, up to measure zero, the (Chapman-Kolmogorov) composition of conditional distributions. That is, for $p$-almost all $x$, and for all measurable $C\subseteq Z$,
$$
\overrightarrow{t\circ s} (C|x) = \int_Y \vec{t}(C|y)\,\vec{s}(dy|x) ,
$$
where the expression on the right-hand side denotes the Lebesgue integral of the measurable function $ \vec{t}(C|-):Y\to\R$ with measure $\vec{s}(-|x)\in PY$. 

\begin{prop}
 Let $X$ be a standard Borel space, and let $c:X\times X\to[0,\infty]$ be a measurable pq-metric (i.e.~measurable as a function). Let $p,q,r\in PX$, and let $s,t$ be couplings of $p$ and $q$ and of $q$ and $r$, respectively. Then
 $$
 \cost_k(t\circ s) \le \cost_k(s) + \cost_k(t) .
 $$
\end{prop}

\begin{proof}
Using Minkowski's inequality,
 \begin{align*}
  \cost(t\circ s) &= \sqrt[k]{ \int_{X\times X} c(x,z)^k \, (t\circ s) (dx\,dz) } \\
  &= \sqrt[k]{ \int_{X\times X\times X} c(x,z)^k \, \cev{s}(dx|y) \, \vec{t}(dz|y) \, q(dy) } \\
  &\le \sqrt[k]{ \int_{X\times X\times X} \big (c(x,y) + c(y,z) \big)^k \, \cev{s}(dx|y) \, \vec{t}(dz|y) \, q(dy) } \\
  &\le \sqrt[k]{ \int_{X\times X} c(x,y)^k \, \cev{s}(dx|y) \, q(dy) } + \sqrt[k]{ \int_{X\times X} c(y,z)^k \, \vec{t}(dz|y) \, q(dy) } \\
  &= \cost(s) + \cost(t). \qedhere
 \end{align*}
\end{proof}

Using this fact we can now show that for each $k\ge 1$, couplings and their costs form a weighted category.

\begin{prop}
 Let $X$ be a standard Borel space, and let $c:X\times X\to[0,\infty]$ be a measurable pq-metric. 
 For each $k$, the set of probability measures $PX$ forms a weighted category, where the arrows are couplings with $k$-cost as their weight. 
\end{prop}

 We will denote the weighted category by $P_k(X,c)$, or more briefly $P_kX$ then this causes no ambiguity. For the unweighted case, we will simply denote the category by $PX$.

\begin{proof}
 The only properties that need to be checked are associativity and unitality of the composition. Note that those need to hold \emph{strictly}, not just almost surely (since the morphisms of $PX$ are joint distributions, not conditional). 
 
 For unitality, let $p,q\in PX$, and let $s\in\Gamma(p,q)$. We have that for all measurable sets $A,B\subseteq X$, using formula \eqref{bayes},
 $$
 (s\circ I_p) (A\times B) = \int_X 1_A(x) \,\vec{s}(B|x)  \, p(dx) = \int_A \vec{s}(B|x) \, p(dx) = s(A\times B) ,
 $$
 and similarly for $I_q\circ s$, so that $s\circ I_p=s=I_q\circ s$.
 
 For associativity, let $p,q,m,n\in PX$, and consider composable couplings $s\in\Gamma(p,q)$, $t\in\Gamma(q,m)$, and $u\in\Gamma(m,n)$. 
 Then for all measurable sets $A,B\subseteq X$, again using \eqref{bayes},
 \begin{align*}
 ((u\circ t)\circ s) (A\times B) &= \int_X \cev{s}(A|y) \, \overrightarrow{u\circ t}(B|y) \, q(dy) \\
 &= \int_X \int_X \cev{s}(A|y) \, \vec{u}(B|z) \, \vec{t}(dz|y) \, q(dy)  \\
 &= \int_X \int_X \cev{s}(A|y) \, \vec{u}(B|z) \, \cev{t}(dy|z)\,m(dz)  \\
 &= \int_X \overleftarrow{t\circ s}(A|z) \, \vec{u}(B|z) \, m(dz) \\
 &= (u\circ (t\circ s)) (A\times B) . \qedhere
 \end{align*}
\end{proof}

Optimization on the arrows gives the following pq-metrics,
\begin{equation}\label{Wasserstein}
c_k(p,q) = \inf_{s\in\Gamma(p,q)} \cost_k(s) = \inf_{s\in\Gamma(p,q)} \sqrt[k]{ \int_{X\times X} c(x,y)\,s(dx\,dy) },
\end{equation}
which are the celebrated Wasserstein pq-metrics. Moreover, the following are true.
\begin{itemize}
 \item If the cost function $c$ on $X$ is symmetric, then $P_kX$ is a symmetric weighted category, and hence \eqref{Wasserstein} is a symmetric cost function as well (i.e.~a pseudometric).
 \item It is well known that the infimum in \eqref{Wasserstein} is actually attained in many cases~\cite[Theorem~4.1]{villani}, and in that case $P_kX$ is an optimization-complete weighted category.
  \item If we see $X$ as a weighted category (with unique morphisms between any two points), the map $\delta:X\to PX$ assigning to each point $x$ the Dirac measure $\delta_x$ is an embedding of weighted categories. The reason why we have an embedding is that any two Dirac measures $\delta_x$ and $\delta_y$ admit a unique coupling $\delta_{(x,y)}$. 
 \item The map $\delta$ is natural in $X$, on objects it can be seen as the unit of the Kantorovich monad \cite{breugel,ours_kantorovich}. Moreover, the multiplication of the Kantorovich monad is a nontrivial example of a submetry, as proven in \cite[Appendix~A]{ours_ordered}.
\end{itemize}

\section{Main statements}\label{mains}

Here are some results involving the constructions defined in the previous sections. In \Cref{pushf} we show that pushforwards of probability measures are functorial (\Cref{funpushf}), and that the pushforward of an embedding is an embedding of categories (both for the weighted and unweighted case). Most importantly, in \Cref{lensestolenses} we show that lenses between the underlying spaces induce (weighted) lenses between the respective categories of couplings, where the liftings are formed by suitable conditional distributions, and once again that such an assignment is functorial (\Cref{funlensestolenses}).

\subsection{Pushforward measures}\label{pushf}

Recall that given a measurable function $f:X\to Y$ and a probability measure $p\in PX$, the \emph{pushforward} of $p$ along $f$ is the probability measure $f_{\sharp}p\in PY$ given by
$$
f_{\sharp}p(B) \coloneqq p(f^{-1}(B))
$$
for all measurable subsets $B\subseteq Y$.
The assignment $p\mapsto f_{\sharp}p$ gives then a function $f_{\sharp}:PX\to PY$. Examples of pushforward measures are given by the marginal projections: the marginalization $P(X\times Y)\to PX$ is the pushforward along the product projection $\pi_1:X\times Y\to X$. 

This assignment is functorial, in the sense that if we denote by $\cat{Meas}$ the category of measurable spaces and functions, the pushforward gives a functor $P:\cat{Meas}\to\cat{Set}$ which maps a measurable set $X$ to the set $PX$ and a measurable function $f:X\to Y$ to its pushforward. One can also equip the sets $PX$ themselves with a measurable structure, and obtain an endofunctor on $\cat{Meas}$, the underlying functor of the Giry monad \cite{giry}.
Since in the previous section we made $PX$ a (weighted) category, we would now like to show that the pushforward $f_\sharp:PX\to PY$ is a (weighted) functor.

Before that, let's look at a useful technical result.

\begin{lemma}\label{pfcond}
 Let $f:X\to Y$ be a measurable function between standard Borel spaces. Let $s\in P(X\times X)$ be a coupling of $p$ and $q\in PX$. A conditional distribution for the pushforward coupling $f^2_{\sharp}s\in P(Y\times Y)$ between $f_{\sharp}p$ and $f_{\sharp}q$ satisfies
 $$
 \overrightarrow{f^2_{\sharp}s}\big(B'|f(x)\big) = \vec{s}\big(f^{-1}(B')|x\big)
 $$
 for $p$-almost all $x\in X$ and all measurable $B'\subseteq Y$.
\end{lemma}

\begin{proof}
 For all measurable subsets $B,B'\subseteq Y$, we have that 
 \begin{align*}
  \int_{f^{-1}(B)} \overrightarrow{f^2_{\sharp}s}\big(B'|f(x)\big) \, p(dx) &= \int_{B} \overrightarrow{f^2_{\sharp}s}(B'|y) \, f_{\sharp}p(dy) \\ 
  &= f^2_{\sharp}s(B\times B') \\
  &= s\big(f^{-1}(B)\times f^{-1}(B')\big) \\
  &= \int_{f^{-1}(B)} \vec{s}\big(f^{-1}(B')|x\big) \, p(dx) . \qedhere
 \end{align*}
\end{proof}

Now let's apply this result to the following proposition.

\begin{prop}\label{pushforward}
 Let $X$ and $Y$ be standard Borel spaces. Let $f:X\to Y$ be a measurable function. 
 The pushforward mapping $f_{\sharp}:PX\to PY$ induces a functor between the respective categories of couplings, where the pushforward of a coupling is given by the map 
 $$
 \begin{tikzcd}
  P(X\times X) \ar{r}{(f\times f)_{\sharp}} & P(Y\times Y) ,
 \end{tikzcd}
 $$
 or more briefly, $f^2_{\sharp}:P(X\times X)\to P(Y\times Y)$.
 
 Moreover, if $X$ and $Y$ are equipped with measurable cost functions, and $f$ is 1-Lipschitz, then $f_\sharp$ is a weighted functor $P_kX\to P_kY$ for all $k$.
\end{prop}

Explicitly, given $s\in P(X\times X)$ and measurable subsets $B,B'\subseteq Y$ the function $f^2_{\sharp}$ gives,
$$
f^2_{\sharp}s(B\times B') = s\big(f^{-1}(B)\times f^{-1}(B')\big) .
$$

\begin{proof}
 For the identity coupling $I_p$, we have that 
 $$
 f^2_{\sharp}I_p (B\times B') = p \big(f^{-1}(B)\intersection f^{-1}(B')) = p\big( f^{-1}(B\intersection B') \big) = I_{f_{\sharp}p}(B\times B')
 $$
 for all measurable subsets $B,B'\subseteq Y$.
 For composition, let $p,q,r\in PX$, let $s$ be a coupling from $p$ to $q$, and let $t$ be a coupling between $q$ and $r$. By \Cref{pfcond},
 \begin{align*}
  f^2_{\sharp}(t\circ s)(B\times B') &= (t\circ s)\big(f^{-1}(B)\times f^{-1}(B')\big) \\
  &= \int_{X} \cev{s}\big(f^{-1}(B)|x\big) \, \vec{t}\big(f^{-1}(B')|x\big) \, q(dx) \\
  &= \int_{X}  \overleftarrow{f^2_{\sharp}s}\big(B|f(x)\big) \, \overrightarrow{f^2_{\sharp}t}\big(B'|f(x)\big) \, q(dx) \\
  &= \int_Y \overleftarrow{f^2_{\sharp}s}(B|y) \, \overrightarrow{f^2_{\sharp}t}(B'|y) \, f_{\sharp}q(dy) \\
  &= (f^2_{\sharp}t)\circ(f^2_{\sharp}s) (B\times B') ,
 \end{align*}
 again for all measurable subsets $B,B'\subseteq Y$. Therefore $f_{\sharp}:PX\to PY$ is functorial.
 
 Suppose now that $X$ and $Y$ are equipped with cost functions $c_X$ and $c_Y$. The ($k$-th power of the) cost of  $f^2_{\sharp}s\in P(Y\times Y)$ is 
 \begin{align*}
 \cost (f^2_{\sharp}s)^k &= \int_{Y\times Y} c_Y(y,y')^k \, f^2_{\sharp}s(dy\,dy') \\
 &= \int_{X\times X} c_Y\big(f(x),f(x')\big)^k \, s(dx\,dx') \\
 &\le \int_{X\times X} c_X(x,x')^k \, s(dx\,dx') \\
 &= \cost (s)^k . \qedhere
 \end{align*}
\end{proof}

It is easy to see that the assignment $f\mapsto f_\sharp$ also preserves identities and composition.
\begin{cor}\label{funpushf}
 Denote by $\cat{SBS}$ the category of standard Borel spaces and measurable maps. 
 The assignment $X\mapsto PX$ and $f\mapsto f_\sharp$ defines a functor $\cat{SBS}\to\cat{Cat}$.
 
 Just as well, denote by $\cat{SBpqMet}$ the category of standard Borel spaces equipped with measurable cost functions and 1-Lipschitz measurable maps. 
 For each $k$, the assignment $X\mapsto P_kX$ and $f\mapsto f_\sharp$ defines a functor $\cat{SBpqMet}\to\cat{WCat}$.
\end{cor}

This assignment $f\mapsto f_\sharp$ also preserves embeddings, as the next proposition shows.

\begin{prop}\label{embedding}
 Let $X$ and $Y$ be standard Borel spaces, and let $i:X\to Y$ be a measurable embedding (i.e.~an injective function, such that $\Sigma(X)$ can be written as the pullback $\sigma$-algebra of $\Sigma(Y)$ along the inclusion $i$). 
 Then the map $i_\sharp:PX\to PY$ is a full embedding of categories.

 If moreover $X$ and $Y$ are equipped with measurable cost functions and $i$ is an isometric embedding, 
 then $i_{\sharp}:P_kX\to P_kY$ is an embedding of weighted categories for every $k$.
\end{prop}

\begin{proof}
 Let $p,q\in PX$. Consider the restriction of $i^2_\sharp:P(X\times X)\to P(Y\times Y)$ to the couplings of $p$ and $q$, call it again $i^2_\sharp:\Gamma(p,q)\to\Gamma(i_\sharp p, i_\sharp q)$. 
 We have to prove that this new mapping is a (weight-preserving) bijection. 
 To prove that it is injective, consider $r,r'\in \Gamma(p,q)$ and suppose that $i^2_\sharp r = i^2_\sharp r'$. Since $i$ is a measurable embedding, for all measurable subsets $A,A'\subseteq X$ we can find subsets $B,B'\subseteq Y$ such that $A=i^{-1}(B)$ and $A'=i^{-1}(B')$. 
 Therefore, for all measurable subsets $A,A'\subseteq X$, 
 $$
 r(A\times A') = r(i^{-1}(B)\times i^{-1}(B')) = i^2_\sharp r(B\times B') = i^2_\sharp r'(B\times B') = r'(A\times A') ,
 $$
 so that $r=r'$. 
 
 To prove surjectivity, let $s$ be a coupling of $i_\sharp p$ and $i_\sharp q$. Since $i$ is a measurable embedding, for each measurable subset $A\subset X$, the (non-inverse) image $i(A)\subseteq Y$ is again measurable. 
 Construct now the coupling $r\in \Gamma(p,q)$ by 
 $$
 r(A\times A') \coloneqq s\big(i(A)\times i(A')\big) ,
 $$
 so that $i^2_\sharp r=s$. Therefore $i^2_\sharp:\Gamma(p,q)\to\Gamma(i_\sharp p, i_\sharp q)$ is a bijection. 
 
 Suppose now that $X$ and $Y$ have cost functions $c_X$ and $c_Y$, and that $i$ is an isometric embedding. Then given $r\in \Gamma(p,q)$, 
 \begin{align*}
 \cost_k (i^2_{\sharp}s)^k &= \int_{Y\times Y} c_Y(y,y')^k \, i^2_{\sharp}r(dy\,dy') \\
 &= \int_{X\times X} c_Y\big(i(x),i(x')\big)^k \, r(dx\,dx') \\
 &= \int_{X\times X} c_X(x,x')^k \, r(dx\,dx') \\
 &= \cost_k(s)^k ,
 \end{align*}
 which means that the bijection is weight-preserving. 
\end{proof}

\subsection{Lenses to lenses}\label{lensestolenses}

We now come to the main result of this work. Namely, we show that measurable lenses give rise to lenses between the categories of couplings, making $X\mapsto PX$ functorial on lenses as well. 
Again, conditional distributions are the key construction.

\begin{thm}\label{main}
 Let $(f,\varphi):X\to Y$ be a measurable lens between standard Borel spaces. There is a categorical lens $(f_\sharp,\tilde\varphi_\sharp)$ between $PX$ and $PY$ where 
 \begin{itemize}
  \item the projection $f_\sharp:PX\to PY$ is the usual pushforward of measures;
  \item the lifting $\tilde\varphi_\sharp:PX \times_{PY} P(Y\times Y)\to P(X\times X)$ takes $p\in PX$ and a coupling $s\in P(Y\times Y)$ whose first marginal has to equal $f_\sharp p$, and returns the coupling $\tilde\varphi_\sharp(p,s)\in P(X\times X)$ given by
  \begin{equation}\label{lenscond}
  \tilde\varphi_\sharp(p,s)(A\times A') \coloneqq \int_A\int_Y 1_{A'}\big(\varphi(x,y)\big) \, \vec{s}\big(dy|f(x)\big) \, p(dx)
  \end{equation}
  for all measurable subsets $A,A'\subseteq X$.
 \end{itemize}
 
 Moreover, if $X$ and $Y$ are equipped with measurable cost functions and $(f,\varphi)$ is a measurable metric lens, $(f_\sharp,\tilde\varphi_\sharp)$ is a weighted lens between $P_kX$ and $P_kY$ for each $k$. 
\end{thm}

In order to better understand the formula \eqref{lenscond}, we can look at its conditional form,
\begin{equation}\label{condcoup}
 \overrightarrow{\tilde\varphi_\sharp(s)}(A|x) = \int_Y 1_{A}\big(\varphi(x,y)\big) \, \vec{s}\big(dy|f(x)\big) .
\end{equation}
First of all, note that this conditional distribution does not depend on $p$ (the joint does). 
Now \eqref{condcoup} says that the probability of transitioning from $x$ to a point of $A$ is intuitively obtained as follows.
\begin{itemize}
 \item For each $y\in Y$, we look at the probability (given by $\vec{s}$) of going from $f(x)$ to $y$.
 \item We lift each point $y\in Y$ to the point $\varphi(x,y)$ as prescribed by the lens $(f,\varphi)$. We can do so measurably.
 \item We construct a kernel that maps $x$ to $\varphi(x,y)$ with the same probability as $s$ maps $f(x)$ to $y$.
 \item We look at the probability, as prescribed by this new kernel, that $\varphi(x,y)$ is in $A$.
\end{itemize}

 Let's now look at the rigorous proof.
\begin{proof}
 We already know by \Cref{pushforward} that $f_\sharp:PX\to PY$ is a functor, weighted if $X$ and $Y$ have cost functions and $f$ is 1-Lipschitz. 
 
 Let us now turn to $\tilde\varphi_\sharp$. We first have to show that for every $p\in PX$ and for every coupling $s\in P(Y\times Y)$ with $f_\sharp p$ as first marginal, the coupling $\tilde\varphi_\sharp(p,s)$ is a lifting of $s$ at $p$. For that, its first marginal needs to be equal to $p$, and its pushforward along $f$ down to $Y\times Y$ needs to be equal to $s$. Now for every measurable $A\subseteq X$,
 \begin{align*}
  \tilde\varphi_\sharp(p,s)(A\times X) = \int_A\int_Y 1 \, \vec{s}\big(dy|f(x)\big) \, p(dx) = \int_A p(dx) = p(A) ,
 \end{align*}
 and for every measurable $B,B'\subseteq Y$,
 \begin{align*}
  (f\times f)_\sharp \tilde\varphi_\sharp(p,s) (B\times B') &= \tilde\varphi_\sharp(p,s)\big(f^{-1}(B)\times f^{-1}(B')\big) \\
   &= \int_{f^{-1(B)}}\int_Y 1_{f^{-1}(B')}\big(\varphi(x,y)\big) \, \vec{s}\big(dy|f(x)\big) \, p(dx) \\
   &= \int_{f^{-1(B)}}\int_Y 1_{B'}\big(f(\varphi(x,y))\big) \, \vec{s}\big(dy|f(x)\big) \, p(dx) \\
   &= \int_{f^{-1(B)}}\int_Y 1_{B'}(y) \, \vec{s}\big(dy|f(x)\big) \, p(dx) \\
   &= \int_{f^{-1(B)}} \vec{s}\big(B'|f(x)\big) \, p(dx) \\
   &= \int_{B} \vec{s}\big(B'|y\big) \, f_\sharp p(dx) \\
   &= s(B\times B') ,
 \end{align*}
 where between the third and fourth line we used that $f(\varphi(x,y))=y$, i.e.~the lifting law of the lens $(f,\varphi)$.
 Therefore $\tilde\varphi_\sharp(p,s)$ is a lifting of $s$ at $p$.
 
 In order to show that that $\tilde\varphi_\sharp$ preserves identity and composition, it is better to use the conditional form of the coupling \eqref{condcoup}.
 Now, to show that the identity is lifted to the identity,
 \begin{align*}
  \overrightarrow{\tilde\varphi_\sharp(p,I_q)}(A|x) = \int_Y 1_{A}\big(\varphi(x,y)\big) \, \delta_{f(x)}(dy) = 1_{A}\big(\varphi(x,f(x))\big) = 1_{A}(x) ,
 \end{align*}
 where we used that $\varphi(x,f(x))=x$, i.e.~the identity law of the lens $(f,\varphi)$. 
 To show the composition property, consider the composable couplings $s,s'\in P(Y\times Y)$. Then
 \begin{align*}
  \int_X \overrightarrow{\tilde\varphi_\sharp(s')}(A'|x') \, \overrightarrow{\tilde\varphi_\sharp(s)}(dx'|x) 
  &= \int_X \int_Y 1_{A'}\big(\varphi(x',y')\big) \, \vec{s'}\big(dy'|f(x')\big) \int_Y \delta_{\varphi(x,y)}(dx') \, \vec{s}\big(dy|f(x)\big) \\
  &= \int_Y 1_{A'}\big(\varphi(\varphi(x,y),y')\big) \int_Y \vec{s'}\big(dy'|f(\varphi(x,y))\big) \, \vec{s}\big(dy|f(x)\big) \\
  &= \int_Y 1_{A'}\big(\varphi(x,y')\big) \int_Y \vec{s'}(dy'|y) \, \vec{s}\big(dy|f(x)\big) \\
  &= \int_Y 1_{A'}\big(\varphi(x,y')\big) \, \overrightarrow{s'\circ s} \big(dy'|f(x)\big) ,
 \end{align*}
 where between the second and third line we used that $\varphi(\varphi(x,y),y')=\varphi(x,y')$, i.e.~the composition law of the lens $(f,\varphi)$, and again its lifting law.
 
 Lastly, suppose that $X$ and $Y$ are equipped with cost functions $c_X$ and $c_Y$. Then the ($k$-th power of the) cost of the lifting is
 \begin{align*}
 \cost_k(\tilde\varphi_\sharp(p,s))^k &= \int_{X\times X} c_X(x,x')^k \, \tilde\varphi_\sharp(p,s)(dx\,dx') \\
 &= \int_{X\times X}\int_Y c_X(x,x')^k \, \delta_{\varphi(x,y)}(dx') \, \vec{s}\big(dy|f(x)\big) \, p(dx) \\
 &= \int_X\int_Y c_X\big(x,\varphi(x,y)\big)^k\,\vec{s}\big(dy|f(x)\big) \, p(dx) \\ 
 &= \int_X\int_Y c_Y\big(f(x),y\big)^k\,\vec{s}\big(dy|f(x)\big) \, p(dx) \\
 &= \int_{Y\times Y} c_Y\big(y',y\big)^k\,\vec{s}\big(dy|y'\big) \, f_\sharp p(dy) \\
 &= \int_{Y\times Y} c_Y\big(y',y\big)^k\, s(dy\,dy') = \cost_k(s)^k ,
 \end{align*}
 where between the third and the fourth line we used that $c_X\big(x,\varphi(x,y)\big) = c_Y\big(f(x),y\big)$, i.e.~the metric lens property of $(f,\varphi)$.
 Therefore the $\tilde\varphi_\sharp(p,s)$ is a weight-preserving lifting, and so $(f_\sharp,\tilde\varphi_\sharp)$ is a weighted lens.
\end{proof}

\begin{cor}
 If $(f,\varphi)$ is a metric, measurable lens between $X$ and $Y$ as in \Cref{main}, then $f:PX\to PY$ is in particular a submetry. 
 This remains true if we drop from $(f,\varphi)$ the composition requirement, as in \Cref{partialconverse}. 
\end{cor}

Just as for the case of pushforwards, the assignment $(f,\varphi)\mapsto (f_\sharp, \tilde\varphi_\sharp)$ preserves identities and composition of lenses. 

\begin{cor}\label{funlensestolenses}
 Denote by $\cat{SBSLens}$ the category of standard Borel spaces and measurable lenses. 
 The assignment $X\mapsto PX$ and $(f,\varphi)\mapsto (f_\sharp,\tilde\varphi_\sharp)$ defines a functor $\cat{SBSLens}\to\cat{CatLens}$.
 
 Just as well, denote by $\cat{SBpqMetLens}$ the category of standard Borel spaces equipped with measurable cost functions and metric, measurable lenses. 
 For each $k$, the assignment $X\mapsto P_kX$ and $(f,\varphi)\mapsto (f_\sharp,\tilde\varphi_\sharp)$ defines a functor $\cat{SBpqMetLens}\to\cat{WCatLens}$.
\end{cor}

Finally, let's see what happens for the case of marginal distributions, the pushforwards along product projections.

\begin{eg}\label{prodlens}
 We have seen in \Cref{prodproj} that product projections admit a canonical lens structure, which is metric in the case of pq-metric spaces. 
 Let's see the explicit construction of \Cref{main} for this case. So again set $X=Y\times Z$, and consider the projection $\pi_1:Y\times Z\to Y$.
 \begin{itemize}
  \item The map $(\pi_1)_\sharp:P(Y\times Z)\to PY$ is just marginalization, it takes a measure on $Y\times Z$ and gives the marginal on $Y$. 
  \item The map $\tilde\varphi_\sharp:P(Y\times Z) \times_{PY} P(Y\times Y)\to P(Y\times Z\times Y\times Z)$ takes a joint $r\in Y\times Z$ and a coupling $s\in Y\times Y$ with common first marginal $p\in PY$, and gives the coupling $\tilde\varphi_\sharp(r,s)$ on $Y\times Z\times Y\times Z$ given as follows. For every measurable $A,A'\subseteq Y$ and $B,B'\subseteq Z$,
  \begin{align*}
  \tilde\varphi_\sharp(r,s) (A\times B\times A'\times B') &= \int_{A\times B}\int_Y 1_{A'\times B'}\big(\varphi((y,z),y')\big) \, \vec{s}\big(dy'|\pi_1(y,z)\big) \, r(dy\,dz) \\
   &= \int_{A\times B}\int_Y 1_{A'\times B'}(y',z) \, \vec{s}(dy'|y) \, r(dy\,dz) \\
   &= \int_{A\times B} 1_{B'}(z) \int_Y 1_{A'}(y') \, \vec{s}(dy'|y) \, r(dy\,dz) \\
   &= \int_{A\times B} 1_{B'}(z) \, \vec{s}(A'|y) \, r(dy\,dz) \\
   &= \int_A \int_B 1_{B'}(z) \, \vec{r}(dz|y)\, \vec{s}(A'|y) \, p(dy) \\
   &= \int_A \vec{r}(B\intersection B'|y) \, \vec{s}(A'|y) \, p(dy) .
  \end{align*}
 \end{itemize}
 In other words, the lifting given by $\tilde{\varphi}_\sharp$ is a coupling that is the identity between the $Z$-components, and is a conditional product of $r$ and $s$ along $p$ in the $Y$-components. 
\end{eg}

\addcontentsline{toc}{chapter}{\bibname}
\bibliographystyle{alpha}
\bibliography{lifting}
 
\end{document}